\documentclass[12pt]{amsart}
\usepackage{amssymb,amsmath,amscd}
\usepackage{mathrsfs}
\usepackage{color}
\newtheorem{thm}{\sc Theorem}[section]
\newtheorem{prop}[thm]{\sc Proposition}

\theoremstyle{definition}
\theoremstyle{definition}
\theoremstyle{definition}
\theoremstyle{definition}
\theoremstyle{definition}
\theoremstyle{definition}

\numberwithin{equation}{section}

\ifx\chitalicit\undefined\let\chitalicit\relax\fi
\ifx\chbfit\undefined\let\chbfit\relax\fi
\usepackage{pifont}
\begin{document}
\title[Conformal Geometry and Composite Membranes]{Conformal Geometry and the Composite Membrane Problem}
\author[S. Chanillo]{SAGUN CHANILLO}
\address{S. Chanillo, Department of Mathematics, Rutgers University, 110 Frelinghuysen Rd., Piscataway, NJ 08854, U.S.A.}
\email{chanillo@math.rutgers.edu}

\renewcommand{\subjclassname}{\textup{2000} Mathematics Subject
Classification}

\begin{abstract}
We show that a certain eigenvalue minimization problem in conformal
classes is equivalent to the composite membrane problem in two
dimensions. New free boundary problems of unstable type arise in
higher dimensions linked to the critical GJMS operator. In dimension
$4$ the critical GJMS operator is exactly the Paneitz operator.
\end{abstract}

\maketitle  We wish to study here a minimization problem for
eigenvalues on Riemann surfaces. Consider then $(\Omega^2,g_0)$ a
surface which is bounded and has a smooth boundary. The surface is
endowed with a metric $g_0$. We consider all metrics $g$ that are in
the same conformal class as $g_0$ and write $g\in [g_0]$. That is
  \begin{equation}\label{metdef}
  g=e^{2u}g_0.
  \end{equation}
We denote the Laplace-Beltrami operator for $g$ by $\Delta_g$. We
will now fix two constraints and consider only those metrics $g$
that satisfy two properties.
\begin{enumerate}
\item There is a constant $A$, such that for all metrics $g$ conformal to $g_0$ through (\ref{metdef}) we have, $||u||_{L^\infty(\Omega)}\leq
A$.
\item We will prescribe the volume of the conformal class, that is
we prescribe $M>0$ so that for all metrics
                       $$\int_\Omega dV_g=\int_\Omega e^{2u} dV_{g_0}=M.$$
\end{enumerate}
All metrics conformal to $g_0$ and satisfying the two constraints
above will be said to lie in class $C$.
 Now under the two constraints above we seek to minimize the first
eigenvalue of $-\Delta_g$ with zero boundary conditions. That is we
seek to minimize:

\begin{equation}\label{minim}
  \inf_{g\in C}\inf_{\phi\in H^1_0(\Omega)}\frac{\int_\Omega |\nabla
  _g\phi|_g^2dV_g}{\int_\Omega |\phi|^2 dV_g}
  \end{equation}

We now have the simple proposition:
\begin{prop}\label{compmem} The minimization problem above is equivalent to the
Composite Membrane Problem.
\end{prop}
\begin{proof} The proof follows by simply noting that the numerator
in the Rayleigh quotient in (\ref{minim}) is a conformal invariant
as we are in two dimensions. That is
           $$\int_\Omega |\nabla_g\phi|_g^2 dV_g=\int_\Omega |\nabla_{g_0}
           \phi|_{g_0}^2 dV_{g_0}.$$
Set $\rho =e^{2u}$, and so the problem (\ref{minim}) by virtue of
the constraints can be re-written as: For given $\lambda=e^{-2A}>0,\
\Lambda=e^{2A}<\infty,\ M>0$ such that $0<\lambda
\leq\rho\leq\Lambda<\infty$ solve the minimization problem,
\begin{equation}\label{changrie}\inf_{\int_\Omega \rho dV_{g_0}=M}\inf_{\phi\in
H^1_0(\Omega)}\frac{\int_\Omega|\nabla_{g_0}\phi|^2\
dV_{g_0}}{\int_\Omega|\phi|^2\rho\
        dV_{g_0}}.
        \end{equation}
In the case our background metric $g_0=dx^2+dy^2$ is the flat
metric, then this last minimization problem (\ref{changrie}) is
exactly the one treated in Theorem 13, \cite{CMP}, i.e. the
Composite Membrane Problem.
\end{proof}

The existence and regularity of the minimization problem associated
with the composite membrane problem is the subject of many articles,
\cite{CMP}, \cite{cm}, \cite{shah}, \cite{blank}, \cite{mw},
\cite{ckenig} and \cite{ckto} among others. Thus these articles now
give complete information about the minimization of eigenvalues in
conformal classes, the existence and regularity of the limit metric
and the associated eigenfunction and more crucially the optimal
$C^{1,1}$ regularity of the minimizing eigenfunction. The limit
metric and thus the associated eigenfunction need not be unique due
to a symmetry breaking phenomena \cite{CMP}. We point out that
unlike a traditional minimization problem for eigenfunctions, we
have to find a minimizing pair $(u_\infty,\phi_\infty)$. We remark
that our regularity results rely on blow-up analysis and so
curvature has no role to play in regularity issues. We may summarize
the results in \cite{CMP}, \cite{cm}, \cite {ckenig} and \cite{ckto}
as applied to eigenvalue minimization (\ref{minim}) in the form of a
theorem. The proofs essentially follow by using $\rho=e^{2u}$. We
also restrict to the case $g_0$ is flat.
\begin{thm} There exists a limit metric $\rho_\infty
  g_0=e^{2u_\infty}g_0$ and an associated limit eigenfunction $\phi_\infty$,
  such that,
\begin{enumerate}
\item
     $$ e^{2u_\infty}=\lambda \chi_D+\Lambda\chi_{D^c}$$
where $D\subseteq\Omega$ and $D^c$ denotes the complement of $D$
in $\Omega$.
\item
 $D$ is a sub-level set of the eigenfunction $\phi_\infty$, that is
there exists $c>0$ such that,
        $$D=\{z\in \Omega|\ \phi_\infty(z)\leq c\}.$$
\item
    The limiting eigenfunction $\phi_\infty$ belongs to
$C^{1,1}(\overline{\Omega})$. In particular $\phi_\infty\in
W^{2,2}(\overline{\Omega})$. The $C^{1,1}$ regularity of
$\phi_\infty$ is optimal.
\item
     $D^c$ has finitely many components, and the free boundary
  $\partial D^c$ consists of finitely many, simple, closed
  real-analytic curves.
\item Due to symmetry breaking the function $u_\infty$ associated to
  the limiting metric and the eigenfunction $\phi_\infty$ is not
necessarily unique. If $\Omega$ is a disk, then $u_\infty$ and the eigenfunction
is unique. Additional hypotheses on convex $\Omega$ does guarantee
uniqueness of $(u_\infty,\phi_\infty)$, see \cite{ckenig}.
\item If $\Omega$ is simply-connected, then $D$ is connected.
\end{enumerate}
\end{thm}
We now pass to a higher dimensional analog of the problem stated
above. This concerns the critical GJMS operator and its conformal
invariance properties. The GJMS hierarchy of conformally invariant
operators was constructed in \cite{gjms} and include the Paneitz
operator and the Yamabe operator.

Specifically we consider $(\Omega^n,g_0)$ and the associated
critical GJMS operator $P^{g_0}_{n/2}$. For us the results proved in
\cite{gjms}, \cite{gz} prove crucial. The operator $P^{g_0}_{n/2}$
has the property that if one considers the metric $g=e^{2u}g_0$,
then the GJMS operator in the new metric $P^g_{n/2}$ satisfies the
relation, (see \cite{gjms} \cite{gz})
\begin{equation}\label{panrel}
       P^g_{n/2}(\phi)= e^{-nu}P^{g_0}_{n/2}(\phi).
\end{equation}
The operator $P^g_{n/2}$ is an elliptic, self-adjoint operator with
leading term $(-\Delta_g)^{n/2}$. So in particular it is fourth
order in dimension $4$. This fourth order operator in dimension $4$
is the Paneitz operator. The operators in odd dimensions are
non-local pseudo-differential operators. One may consult \cite{ndia}
for the role in Conformal Geometry of the fractional operators that
arise.

We have the following proposition whose proof follows the same
scheme as the two dimensional case where instead we use
(\ref{panrel}). The proposition leads to a higher order free
boundary problem involving now the critical GJMS operator. In
dimension $4$ the operator that arises is the Paneitz operator and
for odd $n$ the operator is a fractional operator leading to
fractional free boundary problems which are unstable.

\begin{prop} Consider $(\Omega^n,g_0)$, with metrics $g$ conformal
to $g_0$ via the relation (\ref{metdef}) and which lie in class $C$.
Then the problem,
         $$\inf_{g\in C}\inf_{\phi\in H^{n/2}_0(\Omega)}\frac{\int_\Omega
         P^g_{n/2}\phi\phi dV_g}{\int_\Omega |\phi|^2dV_g},$$
is equivalent to
         $$\inf_{\int_\Omega\rho dV_{g_0}
         =M}\inf_{\phi\in H^{n/2}_0(\Omega)}\frac{\int_\Omega P^{g_0}_{n/2}\phi
         \phi\ dV_{g_0}}{\int_\Omega|\phi|^2\rho dV_{g_0}}.$$
with given $M>0,\ \lambda=e^{-nA}>0, \ \Lambda=e^{nA}<\infty$ and
$0<\lambda\leq\rho\leq\Lambda<\infty$ and where $\rho=e^{nu}$.
\end{prop}
{\bf Acknowledgement:} We thank Robin Graham for several helpful and
clarifying comments regarding the properties of the GJMS operators.
The author was supported in part by NSF grant DMS-0855541.

\bibliographystyle{plain}

\begin{thebibliography}{99}

\bibitem{blank}
Blank, Ivan: Eliminating Mixed Asymptotics in Obstacle Type Free
Boundary Problems, {\it Comm. in PDE} {\bf 29} (2004), 1167-1186;
\bibitem{ndia} Chang, S.-Y. A, and Gonzalez del Mar, M. : Fractional
Laplacian in Conformal Geometry, preprint.

\bibitem{CMP}
Chanillo, S., Grieser, D., Imai, M., Kurata, K., and Ohnishi, I.:
Symmetry Breaking and Other Phenomena in the Optimization of
eigenvalues for Composite Membranes, {\it Comm. in Math. Physics}
{\bf 214} (2000), 315-337;
\bibitem{cm}
Chanillo, S., Grieser, D., and Kurata, K.: The Free Boundary in the
Optimization of Composite Membranes, {\it Contemporary Math. of the
AMS} {\bf 268} (2000), 61-81;

\bibitem{ckenig}
Chanillo, S., and Kenig, C.: Weak Uniqueness and Partial Regularity
in the Composite Membrane problem, {\it J. European Math. Soc.} {\bf
10} (2008), 705-737;


\bibitem{ckto}
Chanillo, S., Kenig, C., and To, T. : Regularity of the Minimizers
in the Composite Membrane Problem in $\mathbb{R}^2$, {\it J.
Functional Analysis}, {\bf 255} (2008), 2299-2320;

\bibitem{gjms}
Graham, C. R., Jenne, R., Mason, L. J., and Sparling, G. A. J., :
Conformally Invariant Powers of the Laplacian I. Existence, {\it J.
London Math. Soc.} {\bf 46}(2), (1992), 557-565.

\bibitem{gz}
Graham, C. R., and Zworski, M. : Scattering Matrix in Conformal
Geometry, {\it Inventionnes Math.,} {\bf 152}, (2003), 89-118.


\bibitem{mw}
Monneau, R. and Weiss, G. S. : An unstable elliptic free boundary
problem arising in solid combustion, {\it Duke Math. J.} {\bf 136}
(2007), 321--341;




\bibitem{shah}
Shahgholian, H. : The Singular Set for the Composite Membrane
problem, {\it Comm. in Math. Physics} {\bf 217} (2007), 93-101;




\end{thebibliography}

\end{document}